\documentclass[12pt]{amsart}
\usepackage[letterpaper, margin=1in]{geometry}
\usepackage{graphicx}
\usepackage{amsmath,amssymb,amsthm,mathtools}
\usepackage[all]{xy}
\usepackage[utf8]{inputenc}
\usepackage[hidelinks]{hyperref}
\usepackage{enumitem}
\usepackage{xcolor}
\usepackage{tikz-cd}
\usepackage{tikz}
\usepackage{import}

\title{Torsors under the universal Jacobian over $\mathcal{M}_g$}

\newtheorem{theorem}{Theorem}

\newtheorem{lemma}{Lemma}
\newtheorem{proposition}{Proposition}
\newtheorem{corollary}{Corollary}

\definecolor{yellow}{rgb}{0.99,0.99,0.80}
\begin{document}

\author{Qixiao MA}
\address{Institute of Mathematical Sciences, ShanghaiTech University, 201210, Shanghai, China}
\email{maqx1@shanghaitech.edu.cn}

\begin{abstract}
We consider the universal family of smooth genus-$g$ curves over $\mathbb{C}$.
We show that for $g\geq3$, every torsor under the relative Jacobian is isomorphic to a connected component of the relative Picard scheme. As a byproduct, we show that $\mathrm{Br}(\mathcal{M}_{3,1})=\mathbb{Z}/2\mathbb{Z}$ over $\mathbb{C}$ and $\mathrm{H}_2(\Gamma_{3,1},\mathbb{Z})=\mathbb{Z}/2\mathbb{Z}$, pinning down the torsion subgroup in \cite[Theorem 1.2]{zbMATH01991000}.

\end{abstract}

\maketitle
\section{Introduction}
Throughout we work over the field of complex numbers $\mathbb{C}$.
\subsection{Main result}Let $\mathcal{M}_{g}$ be the moduli stack of smooth genus-$g$ curves, and let $\pi\colon \mathcal{M}_{g,1}\to \mathcal{M}_{g}$ be the universal family. 
The relative Jacobian $\mathrm{Pic}^0_{\mathcal{M}_{g,1}/\mathcal{M}_{g}}$ is a family of abelian varieties over $\mathcal{M}_g$. We show that:
\begin{theorem}\label{Main_Theorem}For $g\geq3$, we have $$\mathrm{H}^1_{\mathrm{\acute{e}t}}(\mathcal{M}_{g},\mathrm{Pic}^0_{\mathcal{M}_{g,1}/\mathcal{M}_{g}})\cong{\mathbb{Z}}/{(2g-2)\mathbb{Z}}\cdot[\mathrm{Pic}^1_{\mathcal{M}_{g,1}/\mathcal{M}_{g}}].$$
\end{theorem}
This establishes a Franchetta-type result: For $g\geq3$, all torsors under the universal Jacobian arise from the ``obvious'' construction.

\subsection{Sketch of the proof}First, the strong Franchetta theorem \cite[Theorem 5.1]{zbMATH02078578} reduces the computation in Theorem \ref{Main_Theorem} to proving the vanishing $\mathrm{H}^1_{\textrm{\'et}}(\mathcal{M}_g,\mathrm{Pic}_{\mathcal{M}_{g,1}/\mathcal{M}_g})=0$. 

Then we show that for $g\geq3$, the pullback map $\pi^*\colon\mathrm{Br}'(\mathcal{M}_g)\to\mathrm{Br}'(\mathcal{M}_{g,1})$ is surjective. For $g\geq4$ this is known by \cite[Theorem 4.1]{zbMATH08105076}. For $g=3$ we perform an explicit analysis of the Leray spectral sequence, where we use a topological version of the Franchetta conjecture \cite[Theorem 24]{zbMATH01590471}, which builds upon results due to Morita. We leave details of the proof to the end.

Now the Leray spectral sequence for $\pi\colon\mathcal{M}_{g,1}\to\mathcal{M}_g$ further reduces the question to proving the injectivity of the pullback map $\pi^*\colon\mathrm{H}^3_{\textrm{\'et}}(\mathcal{M}_g,\mathbb{Q}_\ell/\mathbb{Z}_\ell)\to \mathrm{H}^3_{\textrm{\'et}}(\mathcal{M}_{g,1},\mathbb{Q}_\ell/\mathbb{Z}_\ell).$ 

Finally, via the comparison theorem and the Teichm\"uller uniformization, this injectivity follows from stability results \cite[Theorem 1.2]{zbMATH06492666} on the homology of mapping class groups of oriented surfaces.

\section{Preparation}
\subsection{The strong Franchetta theorem}For $g\geq3$, the moduli stack $\mathcal{M}_{g,n}$ of $n$-pointed smooth genus-$g$ curves has trivial generic stabilizer, hence the generic fiber of the forgetful morphism between coarse moduli spaces $\pi\colon \overline{M}_{g,n+1}\to\overline{M}_{g,n}$ is a smooth genus-$g$ curve.
\begin{proposition}{\cite[Theorem 5.1]{zbMATH02078578}}\label{strong-franchetta}
Let $k$ be a field, and $g\geq3$ and $n\geq0$ be integers. Let $\eta_n\in\overline{M}_{g,n}$ be the generic point in the moduli space of $n$-pointed stable curves of genus $g\geq3$, and $C=\overline{M}_{g,n+1}$ the tautological curve. Then the marked points $c_1,\cdots,c_n\in C_{\eta_n}$ and the canonical class $K_{C_{\eta_n}}$ freely generate the group of $\eta_n$- points of $\mathrm{Pic}_{C_{\eta_n}/\eta_n}$.
\end{proposition}

\subsection{The Teichm\"uller Uniformization}\label{Teich}
Let $\mathcal{M}_{g,n}^{\mathrm{an}}$ be the analytification of the Deligne--Mumford stack $\mathcal{M}_{g,n}$. The Teichm\"uller uniformization gives a presentation $$\mathcal{M}^{\mathrm{an}}_{g,n}\cong[\mathcal{T}_{g,n}/\Gamma_g^n],$$
where $\mathcal{T}_{g,n}$ is homeomorphic to $\mathbb{C}^{3g-3+n}$, see \cite[Chapter XV]{zbMATH05798333}.

As a consequence, the Hochschild-Serre spectral sequence for singular cohomology $$\mathrm{H}^p(\Gamma_g^n,\mathrm{H}^q(\mathcal{T}_{g,n},\mathbb{Q}_\ell/\mathbb{Z}_\ell))\Rightarrow \mathrm{H}^{p+q}(\mathcal{M}_{g,n}^{\mathrm{an}},\mathbb{Q}_\ell/\mathbb{Z}_\ell)$$
degenerates to one row, giving natural isomorphisms $\mathrm{H}^p(\mathcal{M}_{g,n}^{\mathrm{an}},\mathbb{Q}_\ell/\mathbb{Z}_\ell)\cong\mathrm{H}^p(\Gamma_g^n,\mathbb{Q}_\ell/\mathbb{Z}_\ell).$

\subsection{The Brauer group of $\mathcal{M}_{g,n}$}\label{2.3}Since $\mathcal{M}_{g,n}$ are regular and noetherian Deligne--Mumford stacks, the \'etale cohomology groups $\mathrm{H}_{\textrm{\'et}}^2(\mathcal{M}_{g,n},\mathbb{G}_m)$ are torsion \cite[Theorem 2.5(iii)]{zbMATH07275229}, hence the cohomological Brauer groups $\mathrm{Br}'(\mathcal{M}_{g,n}):=\mathrm{H}^2_{\textrm{\'et}}(\mathcal{M}_{g,n},\mathbb{G}_m)_{\textrm{tors}}$ coincide with the \'etale cohomology groups
$\mathrm{H}^2_{\textrm{\'et}}(\mathcal{M}_{g,n},\mathbb{G}_m).$
\begin{proposition}\label{vanishing-brauer}
Let $g\geq3$ and $n\geq0$ be integers, then $$\mathrm{Br}'(\mathcal{M}_{g,n})=\left\{\begin{array}{cl}0&g\geq4,n\geq0\\ \mathbb{Z}/2\mathbb{Z}&g=3,n=0,1\end{array}\right.$$
Moreover, the pullback map $\pi^*\colon\mathrm{Br}'(\mathcal{M}_3)\to\mathrm{Br}'(\mathcal{M}_{3,1})$ is an isomorphism.
\end{proposition}
\begin{proof}For $g\geq4$, we refer to \cite[Theorem 4.1]{zbMATH07456405}. For $g=3$, by \cite[Theorem 1.1]{zbMATH08105076} we know that $\mathrm{Br}'(\mathcal{M}_3)\cong\mathbb{Z}/2\mathbb{Z}$, and by \cite[Corollary 3.22]{arXiv:2509.09661}, we know that  $\pi^*\colon \mathrm{Br}'(\mathcal{M}_3)\to \mathrm{Br}'(\mathcal{M}_{3,n})$ is injective. We aim to show that $\pi^*\colon\mathrm{Br}'(\mathcal{M}_{3})\to\mathrm{Br}'(\mathcal{M}_{3,1})$ is surjective. Since both sides are torsion, it suffices to show surjectivity on the $m$-torsion parts for all $m>0$.
Comparison of Kummer sequence gives us the commutative diagram
$$\xymatrix{0\ar[r]&\mathrm{Pic}(\mathcal{M}_{3})\otimes\mathbb{Z}/m\mathbb{Z}\ar[r]^-{\delta}\ar[d]^{\pi^*}&\mathrm{H}_{\textrm{\'et}}^2(\mathcal{M}_{3},\mu_m)\ar[r]\ar[d]^{\pi^*}&\mathrm{Br}'(\mathcal{M}_{3})[m]\ar[d]^{\pi^*}\ar[r]&0\\
0\ar[r]&\mathrm{Pic}(\mathcal{M}_{3,1})\otimes\mathbb{Z}/m\mathbb{Z}\ar[r]^-{\delta}&\mathrm{H}^2_{\textrm{\'et}}(\mathcal{M}_{3,1},\mu_m)\ar[r]&\mathrm{Br}'(\mathcal{M}_{3,1})[m]\ar[r]&0}$$
By the snake lemma, it suffices to show that $$\mathrm{H}_{\textrm{\'et}}^2(\mathcal{M}_{3,1},\mu_m)=\pi^*\mathrm{H}_{\textrm{\'et}}^2(\mathcal{M}_{3},\mu_m)+\delta\left(\mathrm{Pic}(\mathcal{M}_{3,1})\right).$$
The Leray spectral sequence for $\mu_m$ along $\pi\colon\mathcal{M}_{3,1}\to\mathcal{M}_{3}$ induces a filtration $F^0\supseteq F^1\supseteq F^2$, where $F^0=\mathrm{H}^2_{\textrm{\'et}}(\mathcal{M}_{3,1},\mu_m)$ and $F^2=\pi^*\mathrm{H}^2_{\textrm{\'et}}(\mathcal{M}_{3},\mu_m).$ It suffices to show that any class in $\mathrm{H}^2_{\textrm{\'et}}(\mathcal{M}_{3,1},\mu_m)$ can be modified by $\delta(\mathrm{Pic}(\mathcal{M}_{3,1}))$ so that the image in $F^0/F^1$ and $F^1/F^2$ both equal to zero. We leave the  details to Proposition \ref{add} in Section \ref{finish-proof}.
\end{proof}

\subsection{Homology stability theorem} 
Let $S_{g}$ be a surface of genus $g$, and let $x\in S_g$ be a fixed point. Let $S_g^1=S_g-\{x\}$ be the punctured surface. Let $S_{g,1}$ be the surface with boundary, obtained from $S_g$ by removing an open disc around $x$. We have natural inclusions 
$$S_{g,1}\subseteq S_g^1\subseteq S_g.$$

Let $\Gamma_g=\pi_0\mathrm{Diff}^+(S_g)$ be the group of components of the orientation preserving diffeomorphism group of $S_g$. Let $\Gamma_{g,1}=\pi_0\mathrm{Diff}^+(S_{g,1};\partial S_{g,1})$ and $\Gamma_{g}^1=\pi_0\mathrm{Diff}^+(S_g;\{x\})$ be the group of components of diffeomorphisms that are identity on the boundary of $S_{g,1}$ and on the point $\{x\}$, respectively. 

Let $\delta_g\colon\Gamma_{g,1}\to\Gamma_{g}$ be the map induced by the inclusion $S_{g,1}\hookrightarrow S_g$. Then we have:
\begin{proposition}{\cite[Theorem 1.2]{zbMATH06492666}}
The map $\mathrm{H}_k(\delta_g)\colon\mathrm{H}_k(\Gamma_{g,1},\mathbb{Z})\to\mathrm{H}_k(\Gamma_{g},\mathbb{Z})$ is surjective for $k\leq \frac{2}{3}g+1$ and an isomorphism for $k\leq\frac{2}{3}g$.
\end{proposition}
Let $\epsilon_g\colon\Gamma_g^1\to\Gamma_g$ be the map induced by the inclusion $S_g^1\subseteq S_g$.  Since $\mathrm{H}_k(\delta_g)$ factors through $\mathrm{H}_k(\epsilon_g)\colon \mathrm{H}_k(\Gamma_{g}^1,\mathbb{Z})\to\mathrm{H}_k(\Gamma_g,\mathbb{Z})$, the surjectivity of $\mathrm{H}_k(\delta_g)$ implies the surjectivity of $\mathrm{H}_k(\epsilon_g)$. Therefore we have:
\begin{corollary}\label{Wahl-corollary}
The map $\mathrm{H}_k(\epsilon_g)\colon \mathrm{H}_k(\Gamma_g^1,\mathbb{Z})\to\mathrm{H}_k(\Gamma_{g},\mathbb{Z})$ is surjective for $k\leq\frac{2}{3}g+1$.
\end{corollary}
Taking \(k=3\), we see that
$
\mathrm{H}_3(\epsilon_g)\colon
\mathrm{H}_3(\Gamma_g^1,\mathbb Z)
\twoheadrightarrow
\mathrm{H}_3(\Gamma_g,\mathbb Z)
$ for all $g\geq3$.

\section{Proof of Theorem \ref{Main_Theorem}}

\subsection{Reduction to vanishing}
Consider the short exact sequence of \'etale sheaves on $\mathcal{M}_{g}$
$$\xymatrix{0\ar[r]&\mathrm{Pic}^0_{\mathcal{M}_{g,1}/\mathcal{M}_{g}}\ar[r]&\mathrm{Pic}_{\mathcal{M}_{g,1}/\mathcal{M}_{g}}\ar[r]&\underline{\mathbb{Z}}_{\mathcal{M}_{g}}\ar[r]&0},$$from which we obtain the long exact sequence $$\xymatrix{\mathrm{H}^0(\mathcal{M}_{g},\underline{\mathbb{Z}}_{\mathcal{M}_{g}})\ar[r]^-{\delta}&\mathrm{H}^1_{\textrm{\'et}}(\mathcal{M}_{g},\mathrm{Pic}^0_{\mathcal{M}_{g,1}/\mathcal{M}_{g}})\ar[r]&\mathrm{H}^1_{\textrm{\'et}}(\mathcal{M}_{g},\mathrm{Pic}_{\mathcal{M}_{g,1}/\mathcal{M}_{g}})\ar[r]&\mathrm{H}^1_{\textrm{\'et}}(\mathcal{M}_{g},\underline{\mathbb{Z}}_{\mathcal{M}_{g}})=0.}$$
Here the vanishing of $\mathrm{H}^1_{\textrm{\'et}}(\mathcal{M}_{g},\underline{\mathbb{Z}}_{\mathcal{M}_{g}})$ is due to \cite[Proposition 2.4.2]{zbMATH07384449}. Keeping track of the representing cocycles, the connecting homomorphism $\delta$ sends the element $1\in\mathrm{H}^0(\mathcal{M}_{g},\underline{\mathbb{Z}}_{\mathcal{M}_{g}})$ to the class of the torsor $[\mathrm{Pic}^1_{\mathcal{M}_{g,1}/\mathcal{M}_{g}}]$. 

By Proposition \ref{strong-franchetta}, there are rational points on $\mathrm{Pic}^d_{\mathcal{M}_{g,1}/\mathcal{M}_{g}}$ if and only if $d$ is a multiple of $2g-2$, hence the period of $[\mathrm{Pic}^1_{\mathcal{M}_{g,1}/\mathcal{M}_{g}}]$ equals $2g-2$. 
Therefore, showing that $$\mathrm{H}^1_{\textrm{\'et}}(\mathcal{M}_{g},\mathrm{Pic}^0_{\mathcal{M}_{g,1}/\mathcal{M}_{g}})\cong\mathbb{Z}/(2g-2)\mathbb{Z}\cdot[\mathrm{Pic}^1_{\mathcal{M}_{g,1}/\mathcal{M}_{g}}],$$ is equivalent to showing $$\mathrm{H}_{\textrm{\'et}}^1(\mathcal{M}_{g},\mathrm{Pic}_{\mathcal{M}_{g,1}/\mathcal{M}_{g}})=0.$$

\subsection{Reduction to injectivity}\label{sequence}
We consider the Leray spectral sequence of $\mathbb{G}_m$ along the projection $\pi\colon \mathcal{M}_{g,1}\to \mathcal{M}_{g}$. Since $\pi$ is a family of smooth curves over a smooth base, we have the following lemma:
\begin{lemma}\label{higher-vanishing}We have $\mathrm{R}^i\pi_*\mathbb{G}_m=0$ for all $i\geq2$.
\end{lemma}
\begin{proof}By the Kummer sequence on $\mathcal{M}_{g,1}$, for any $n$ we have the following short exact sequence of direct image sheaves on $\mathcal{M}_g$
$$\xymatrix{0\ar[r]& \mathrm{R}^i\pi_*\mathbb{G}_m\otimes\mathbb{Z}/n\mathbb{Z}\ar[r]^-{\delta}&\mathrm{R}^{i+1}\pi_*\mu_n\ar[r]&\mathrm{R}^{i+1}\pi_*\mathbb{G}_m[n]\ar[r]&0}.$$
Since $\mathcal{M}_{g,1}$ is regular noetherian, by \cite[Lemma 3.5.3]{zbMATH07384449} we know that for $i\geq1$, the sheaves $\mathrm{R}^{i+1}\pi_*\mathbb{G}_m$ are torsion, therefore it suffices to show $\textrm{R}^{i+1}\pi_*\mathbb{G}_m[n]=0$ for all $n$. 

For $i=1$, we know $\mathrm{R}^1\pi_*\mathbb{G}_m$ is the relative Picard sheaf, and the connecting homomorphism $\delta$ is the degree map. Since one can always find relative degree-$1$ line bundles after \'etale localization, the degree map is stalkwise surjective, so $\mathrm{R}^2\pi_*\mathbb{G}_m[n]=0$. For $i\geq2$, we know $\mathrm{R}^{i+1}\pi_*\mu_n=0$ by proper base change, hence $\mathrm{R}^{i+1}\pi_*\mathbb{G}_m[n]=0$.
\end{proof}
With this lemma, we see that the Leray spectral sequence $$\mathrm{H}^p_{\textrm{\'et}}(\mathcal{M}_g,\mathrm{R}^q\pi_*\mathbb{G}_m)\Rightarrow\mathrm{H}_{\textrm{\'et}}^{p+q}(\mathcal{M}_{g,1},\mathbb{G}_m)$$ is concentrated in two rows, and gives rise to the following long exact sequence
$$\xymatrix{\mathrm{H}^2_{\textrm{\'et}}(\mathcal{M}_{g},\mathbb{G}_m)\ar[r]^{\pi^*}&\mathrm{H}_{\textrm{\'et}}^2(\mathcal{M}_{g,1},\mathbb{G}_m)\ar[r]&\mathrm{H}_{\textrm{\'et}}^1(\mathcal{M}_{g},\mathrm{Pic}_{\mathcal{M}_{g,1}/\mathcal{M}_{g}})\\
\ar[r]^-{\mathrm{d}_2^{1,1}}&\mathrm{H}^3_{\textrm{\'et}}(\mathcal{M}_{g},\mathbb{G}_m)\ar[r]^-{\pi^*}&\mathrm{H}^3_{\textrm{\'et}}(\mathcal{M}_{g,1},\mathbb{G}_m)}$$

By Proposition \ref{vanishing-brauer}, the vanishing of $\mathrm{H}^1_{\textrm{\'et}}(\mathcal{M}_g,\mathrm{Pic}_{\mathcal{M}_{g,1}/\mathcal{M}_g})$ is equivalent to the injectivity of $$\pi^*\colon \mathrm{H}^3_{\textrm{\'et}}(\mathcal{M}_g,\mathbb{G}_m)\to\mathrm{H}^3_{\textrm{\'et}}(\mathcal{M}_{g,1},\mathbb{G}_m)$$ for $g\geq3$.
By \cite[Lemma 3.5.3]{zbMATH07384449}, the higher \'etale cohomology groups are torsion, therefore it suffices to show the injectivity on $\ell^\infty$-torsion part for all primes $\ell$: $$\pi^*\colon \mathrm{H}^3_{\textrm{\'et}}(\mathcal{M}_g,\mathbb{G}_m)[\ell^\infty]\to\mathrm{H}^3_{\textrm{\'et}}(\mathcal{M}_{g,1},\mathbb{G}_m)[\ell^\infty].$$ 

\subsection{Applying the stability result}
By taking the direct limit of the Kummer sequences, we have a commutative diagram of short exact sequences
$$\xymatrix{0\ar[r]&\mathrm{H}^2_{\mathrm{\acute{e}t}}(\mathcal{M}_{g},\mathbb{G}_m)\otimes\mathbb{Q}_\ell/\mathbb{Z}_\ell\ar[d]^{\pi^*\otimes\mathbb{Q}_\ell/\mathbb{Z}_\ell}\ar[r]&\mathrm{H}^3_{\mathrm{\acute{e}t}}(\mathcal{M}_{g},\mu_{\ell^\infty})\ar[d]^{\pi^*}\ar[r]&\mathrm{H}^3_{\mathrm{\acute{e}t}}(\mathcal{M}_{g},\mathbb{G}_m)[\ell^\infty]\ar[d]^{\pi^*}\ar[r]&0\\
0\ar[r]&\mathrm{H}^2_{\mathrm{\acute{e}t}}(\mathcal{M}_{g,1},\mathbb{G}_m)\otimes\mathbb{Q}_\ell/\mathbb{Z}_\ell\ar[r]&\mathrm{H}^3_{\mathrm{\acute{e}t}}(\mathcal{M}_{g,1},\mu_{\ell^\infty})\ar[r]&\mathrm{H}^3_{\mathrm{\acute{e}t}}(\mathcal{M}_{g,1},\mathbb{G}_m)[\ell^\infty]\ar[r]&0
}.$$
Applying the snake lemma to the above diagram yields a short exact sequence
$$\xymatrix{\mathrm{ker}(\pi^*|_{\mathrm{H}^3(\mathcal{M}_g,\mu_{\ell^\infty})})\ar[r]&\mathrm{ker}(\pi^*|_{\mathrm{H}^3(\mathcal{M}_g,\mathbb{G}_m)[\ell^\infty]})\ar[r]&\mathrm{coker}(\pi^*|_{\mathrm{H}^2(\mathcal{M}_g,\mathbb{G}_m)}\otimes\mathbb{Q}_\ell/\mathbb{Z}_\ell)}.$$
The third term vanishes by Proposition \ref{vanishing-brauer}, therefore we reduce our problem to showing the injectivity of$$\pi^*\colon\mathrm{H}^3_{\textrm{\'et}}(\mathcal{M}_g,\mu_{\ell^\infty})\to\mathrm{H}^3_{\textrm{\'et}}(\mathcal{M}_{g,1},\mu_{\ell^\infty}).$$

By the comparison theorem for the \'etale and complex topologies \cite[Theorem 3.12]{zbMATH03674235}, after fixing roots of unity, we may identify $\mathrm{H}^3_{\textrm{\'et}}(\mathcal{M}_{g,n},\mu_{\ell^\infty})$ with $\mathrm{H}^3(\mathcal{M}^{\mathrm{an}}_{g,n},\mathbb{Q}_\ell/\mathbb{Z}_\ell)$.
This further reduces Theorem \ref{Main_Theorem} to proving the injectivity of 
the pullback map on singular cohomology $$\pi^*\colon \mathrm{H}^3(\mathcal{M}_g^{\mathrm{an}},\mathbb{Q}_\ell/\mathbb{Z}_\ell)\to\mathrm{H}^3(\mathcal{M}_{g,1}^{\mathrm{an}},\mathbb{Q}_\ell/\mathbb{Z}_\ell).$$

Since $\mathbb{Q}_\ell/\mathbb{Z}_\ell$ is an injective module, by the universal coefficient theorem, for any topological space $X$, we have natural isomorphism $$\mathrm{H}^3(X,\mathbb{Q}_\ell/\mathbb{Z}_\ell)\cong\mathrm{Hom}(\mathrm{H}_3(X,\mathbb{Z}),\mathbb{Q}_\ell/\mathbb{Z}_\ell).$$
Thus the injectivity of $\pi^*$ will be implied by the surjectivity of $$\pi_*\colon \mathrm{H}_3(\mathcal{M}^{\mathrm{an}}_{g,1},\mathbb{Z})\to\mathrm{H}_3(\mathcal{M}_{g}^{\mathrm{an}},\mathbb{Z}).$$ By the discussion in Section \ref{Teich}, this is equivalent to the surjectivity of $$\mathrm{H}_3(\epsilon_g)\colon\mathrm{H}_3(\Gamma_g^1,\mathbb{Z})\to\mathrm{H}_3(\Gamma_g,\mathbb{Z}),$$ which follows from Corollary \ref{Wahl-corollary}. 

\section{Finishing the proof of Proposition \ref{vanishing-brauer}}\label{finish-proof}
\subsection{Some topological results}
For finite locally constant sheaves, we identify \'etale cohomology with complex cohomology. Let us fix a smooth genus-$g$ curve $C$ and represent the local system $\mathrm{R}^i\pi_*\mu_m$ on $\mathcal{M}_g$ via the $\Gamma_g$ action on $\mathrm{H}^i(C,\mu_m)$. Poincar\'e duality naturally identifies  $\mathrm{H}^1(C,\mathbb{Z})$ with  $\mathrm{H}_1(C,\mathbb{Z})$, we denote both by $\mathbb{H}$. 

As remarked in \cite[Corollary 22]{zbMATH01590471}, the base and fiber of $\mathrm{Pic}^d_{\mathcal{M}_{g,1}/\mathcal{M}_{g}}$ are Eilenberg-MacLane spaces, so is the total space, giving a short exact sequence of fundamental groups $$\xymatrix{0\ar[r]&\mathbb{H}\ar[r]&\pi_1(\mathrm{Pic}^d_{\mathcal{M}_{g,1}/\mathcal{M}_g})\ar[r]&\Gamma_g\ar[r]&1},$$ which corresponds to an extension class $\epsilon(d)$ in $\mathrm{H}^2(\Gamma_g,\mathbb{H})$. 

\begin{lemma}\label{identify}Under the identification $\mathrm{H}^2(\Gamma_g,\mathbb{H})\cong\mathrm{H}^2(\mathcal{M}_g,\mathrm{R}^1\pi_*\mathbb{Z})$, up to a $\pm$ sign which depends on convention, this class satisfies $\epsilon(1)=\mathrm{d}_{2,\mathbb{Z}}^{0,1}(1)$, where $\mathrm{d}_{2,\mathbb{Z}}^{0,1}\colon\mathrm{H}^0(\mathcal{M}_g,\mathrm{R}^2\pi_*\mathbb{Z})\to\mathrm{H}^2(\mathcal{M}_g,\mathrm{R}^1\pi_*\mathbb{Z})$ is the transgression map.
\end{lemma}
\begin{proof}The transgression map can be calculated by \v{C}ech cohomology with respect to the Teichm\"uller covering $\{\mathcal{T}_{g}\to\mathcal{M}_{g}\}$. The \v{C}ech nerve $T_{g}^{\times_{\mathcal{M}_g}k+1}\cong(\Gamma_g)^k\times\mathcal{T}_g$ and the \v{C}ech complex is isomorphic to the bar complex for group cohomology of $\Gamma_g$.
\end{proof}

\begin{lemma}{\cite[Theorem 24]{zbMATH01590471}}\label{Hain-lemma} For $g\geq3$, the map $\mathbb{Z}/(2g-2)\mathbb{Z}\to\mathrm{H}^2(\Gamma_g,\mathbb{H})$ that takes $\overline{d}$ to the class $\epsilon(d)$ corresponding to $\mathrm{Pic}^d_{\mathcal{M}_{g,1}/\mathcal{M}_g}$ is a group isomorphism.
\end{lemma}

\subsection{The key statement} In Section \ref{2.3}, Proposition \ref{vanishing-brauer} was reduced to the statement that $$\mathrm{H}^2(\mathcal{M}_{g,1},\mu_m)=\pi^*\mathrm{H}^2(\mathcal{M}_g,\mu_m)+\delta(\mathrm{Pic}(\mathcal{M}_{g,1})).$$
We prove, independent of \cite[Theorem 4.1]{zbMATH07456405}, that the statement holds for all $g\geq3$.
\begin{proposition}\label{add}For the Leray spectral sequence of $\mu_m$ along $\pi\colon \mathcal{M}_{g,1}\to\mathcal{M}_g$, we have:
\begin{enumerate}
\item $E_r^{0,2}$ stabilizes on the $E_3$ page, and equals the subgroup of $\mathrm{H}^0(\mathcal{M}_{g},\mathrm{R}^2\pi_*\mu_m)$ generated by the class of the relative canonical bundle $\mathcal{L}:=\omega_{\mathcal{M}_{g,1}}\otimes\pi^*\omega_{\mathcal{M}_{g}}^{\otimes-1}$.
\item $E_r^{1,1}$ stabilizes on the $E_2$ page, and $\mathrm{H}^1(\mathcal{M}_g,\mathrm{R}^1\pi_*\mu_m)\cong\mathbb{Z}/(m,2g-2)\mathbb{Z}$. Furthermore, the group is generated by the image of $\mathcal{L}^{\otimes \frac{m}{(m,2g-2)}}$ in $F^1/F^2=E_\infty^{1,1}$.
\end{enumerate}
Therefore any class in $\mathrm{H}^2(\mathcal{M}_{g,1},\mu_m)$ can be modified to $\pi^*\mathrm{H}^2(\mathcal{M}_{g},\mu_m)$ by adding suitable multiple of $\mathcal{L}$, hence $\mathrm{H}^2(\mathcal{M}_{g,1},\mu_m)=\pi^*\mathrm{H}^2(\mathcal{M}_{g},\mu_m)+\delta(\mathrm{Pic}(\mathcal{M}_{g,1}))$.
\end{proposition}
\begin{proof} 
(1) In order to determine $E_r^{0,2}$, we first look at the transgression map $$\mathrm{d}_2^{0,2}\colon \mathrm{H}^0(\mathcal{M}_g,\mathrm{R}^2\pi_*\mu_m)\to \mathrm{H}^2(\mathcal{M}_g,\mathrm{R}^1\pi_*\mu_m),$$ 
In general the map is given by taking Yoneda product with the class $c\in\mathrm{Ext}^2(\mathrm{R}^2\pi_*\mu_m,\mathrm{R}^1\pi_*\mu_m)$ given via the triangle $$\xymatrix{\mathrm{R}^1\pi_*\mu_m\ar[r]& \tau_{[1,2]}\mathrm{R}\pi_*\mu_m\ar[r]&\mathrm{R}^2\pi_*\mu_m[-1]\ar[r]^-{c}& \mathrm{R}^1\pi_*\mu_m[1]}.$$Since Kummer sequence shows that $\mu_m$ is quasi-isomorphic to the complex $[\xymatrix{\mathbb{G}_m\ar[r]^-{m}&\mathbb{G}_m}]$ concentrated in degrees $0,1$, using Lemma \ref{higher-vanishing}, we see that the transgression map can be obtained by taking cup product with the extension class of
$$\xymatrix{0\ar[r]&\mathrm{R}^1\pi_*\mu_m\ar[r]&\mathrm{R}^1\pi_*\mathbb{G}_m\ar[r]^-m&\mathrm{R}^1\pi_*\mathbb{G}_m\ar[r]&\mathrm{R}^2\pi_*\mu_m\ar[r]&0}.$$

We wish to identify this map with the topological description in Lemma \ref{Hain-lemma}.
Choose a primitive $m$-th root of unity $\zeta_m$ and consider the commutative diagram of analytic sheaves on $\mathcal{M}_{g,1}$
$$\xymatrix{0\ar[r]&\mathbb{Z}\ar[r]\ar[d]&\mathcal{O}\ar[d]_{\mathrm{exp}}\ar[r]^-{\mathrm{exp}}&\mathcal{O}^*\ar[d]^-m\ar[r]&0\\
0\ar[r]&\mathbb{Z}/m\mathbb{Z}\cdot\zeta_m\ar[r]&\mathcal{O}^*\ar[r]^-m&\mathcal{O}^*\ar[r]&0}$$
it yields the commutative diagram of direct images
$$\xymatrix{\mathrm{R}^1\pi_*\mathbb{Z}\ar[d]\ar[r]&\mathrm{R}^1\pi_*\mathcal{O}\ar[r]^-{\mathrm{exp}}\ar[d]_-{\mathrm{exp}}&\mathrm{R}^1\pi_*\mathcal{O}^*\ar[d]^-m\ar[r]^\delta&\mathrm{R}^2\pi_*\mathbb{Z}\ar[d]\\
\mathrm{R}^1\pi_*\mathbb{Z}/m\mathbb{Z}\ar[r]&\mathrm{R}^1\pi_*\mathcal{O}^*\ar[r]^-{m}&\mathrm{R}^1\pi_*\mathcal{O}^*\ar[r]^-\delta&\mathrm{R}^2\pi_*\mathbb{Z}/m\mathbb{Z}.}$$
Naturality of the Yoneda pairing gives us commutative diagram
$$\xymatrix{\mathrm{H}^0(\mathcal{M}_g,\mathrm{R}^2\pi_*\mathbb{Z})\ar[r]^-{\mathrm{d}_{2,\mathbb{Z}}^{0,2}}\ar[d]^\alpha&\mathrm{H}^2(\mathcal{M}_g,\mathrm{R}^1\pi_*\mathbb{Z})\ar[d]^{\beta}\ar@{=}[r]&\mathrm{H}^1(\Gamma_g,\mathbb{H})\cong\frac{\mathbb{Z}}{(2g-2)\mathbb{Z}}\cdot\epsilon(1)\\
\mathrm{H}^0(\mathcal{M}_g,\mathrm{R}^2\pi_*\left(\mathbb{Z}/m\mathbb{Z}\right))\ar[r]^-{\mathrm{d}_2^{0,2}}&\mathrm{H}^2(\mathcal{M}_g,\mathrm{R}^1\pi_*\left(\mathbb{Z}/m\mathbb{Z}\right))\ar@{=}[r]&\mathrm{H}^2(\Gamma_g,\mathbb{H}/m\mathbb{H})}$$
the left vertical map $\alpha$ is quotient by $m$. 
Running the Bockstein sequence $$\xymatrix{0\ar[r]&\mathrm{R}^1\pi_*\mathbb{Z}\ar[r]^-{m}&\mathrm{R}^1\pi_*\mathbb{Z}\ar[r]&\mathrm{R}^1\pi_*\left(\mathbb{Z}/m\mathbb{Z}\right)\ar[r]&0}$$
we see that $\beta$ factors through the natural injection $$i\colon\mathrm{H}^2(\mathcal{M}_g,\mathrm{R}^1\pi_*\mathbb{Z})\otimes \left(\mathbb{Z}/m\mathbb{Z}\right)\to\mathrm{H}^2(\mathcal{M}_g,\mathrm{R}^1\pi_*(\mathbb{Z}/m\mathbb{Z})).$$
By Lemma \ref{Hain-lemma} and Lemma \ref{identify}, under the identification $$\mathrm{H}^2(\mathcal{M}_g,\mathrm{R}^1\pi_*\mathbb{Z})\otimes(\mathbb{Z}/m\mathbb{Z})\cong\mathrm{H}^2(\Gamma_g,\mathbb{H})\otimes(\mathbb{Z}/m\mathbb{Z})\cong\mathbb{Z}/(m,2g-2)\mathbb{Z}\cdot\epsilon(1),$$ the order of  $\mathrm{d}_2^{0,2}(\overline{1})$ equals $(m,2g-2)$, hence $\mathrm{ker}(\mathrm{d}_2^{0,2})$ is the subgroup of $\mathrm{H}^0(\mathcal{M}_{g},\mathrm{R}^2\pi_*\mu_m)=\mathbb{Z}/m\mathbb{Z}$ generated by the class $\overline{2g-2}$. This class is represented by the relative canonical bundle, and survives to the $E_\infty$ page, therefore $E_{r\geq3}^{0,2}$ stabilizes.

(2) Since $\deg\mathcal{L}=2g-2$, the degree of $\mathcal{L}^{\otimes\frac{m}{(m,2g-2)}}$ is divisible by $m$, so its chern class $\eta=c_1(\mathcal{L}^{\otimes\frac{m}{(m,2g-2)}})$ lies in $F^1$. We aim to show that the image $\overline{\eta}$ generates $F^1/F^2\subseteq \mathrm{H}^1(\mathcal{M}_g,\mathrm{R}^1\pi_*(\mathbb{Z}/m\mathbb{Z}))$. 

As argued in \cite[Proposition 23]{zbMATH01590471}, for $g\geq2$, we have $\mathrm{H}^1(\Gamma_g,\mathbb{H})=0$. Therefore the boundary map induced by the Bockstein sequence gives an isomorphism  $$\delta\colon\mathrm{H}^1(\mathcal{M}_g,\mathrm{R}^1\pi_*(\mathbb{Z}/m\mathbb{Z}))\cong\mathrm{H}^2(\mathcal{M}_g,\mathrm{R}^1\pi_*\mathbb{Z})[m].$$
By Lemma \ref{Hain-lemma}, the group $\mathrm{H}^2(\mathcal{M}_g,\mathrm{R}^1\pi_*\mathbb{Z})[m]=\mathrm{H}^2(\Gamma_g,\mathbb{H})[m]$ is generated by $\frac{2g-2}{(m,2g-2)}\cdot\epsilon(1)$, where $\epsilon(1)=\mathrm{d}_2^{0,1}(1)$ up to a sign. Since $\eta$ is the cohomology class of a global line bundle, it survives to the $E_\infty$ page. Therefore it suffices to show that $$\delta(\eta)=\frac{2g-2}{(m,2g-2)}\cdot\mathrm{d}_{2,\mathbb{Z}}^{0,1}(1).$$
\begin{itemize}
\item First we calculate $\eta=c_1(\mathcal{L}^{\otimes\frac{2g-2}{(2g-2,m)}})$, using that the chern class is the connecting homomorphism of the Kummer sequence.
We pick an $m$-th root line bundle $\mathcal{N}_i\in\pi^{-1}(U_i)$ for $\mathcal{L}^{\otimes\frac{m}{(m,2g-2)}}$ on some open covering $\{U_i\}_{i\in I}$. This yields a 0-cocycle. Since $\mathcal{N}_i^{\otimes m}|_{U_{ij}}\cong\mathcal{L}|_{U_{ij}}\cong\mathcal{N}_{j}^{\otimes m}|_{U_{ij}}$, after taking coboundary, the \v{C}ech $1$-cocycle takes values in $\mathrm{Pic}^0_{\mathcal{M}_{g,1}/\mathcal{M}_g}[m]$.
\item We identify $\mathrm{Pic}_{\mathcal{M}_{g,1}/\mathcal{M}_g}^0[m]$ with $\mathrm{R}^1\pi_*(\mathbb{Z}/m\mathbb{Z})$ using the exact sequence of topological uniformization $$\xymatrix{0\ar[r]& \mathrm{R}^1\pi_*\mathbb{Z}\ar[r]&\mathrm{R}^1\pi_*\mathbb{R}\cong R^1\pi_*\mathcal{O}\ar[r]&\mathrm{Pic}^0_{\mathcal{M}_{g,1}/\mathcal{M}_g}\ar[r]&0},$$
where the identification in the middle is obtained via Hodge theory using inclusion and projection $\mathrm{R}^1\pi_*\mathbb{R}\to\mathrm{R}^1\pi_*\mathbb{C}\to\mathrm{R}^1\pi_*\mathcal{O}$.
For any local section $h_i$ of $\mathrm{Pic}^0_{\mathcal{M}_{g,1}/\mathcal{M}_g}[m]$, we lift it to $\widetilde{h}_i$ on $\mathrm{R}^1\pi_*\mathbb{R}$. Since $mh_i=0$, we know that $m\widetilde{h}_i$ is a local section of $\mathrm{R}^1\pi_*\mathbb{Z}$. Note that for different choices of $h_i$, the section $m\widetilde{h}_i$ differ by $m\mathrm{R}^1\pi_*\mathbb{Z}$, therefore the identification $\mathrm{Pic}^0_{\mathcal{M}_{g,1}/\mathcal{M}_g}\to\mathrm{R}^1\pi_*(\mathbb{Z}/m\mathbb{Z}),\ h_i\mapsto m\widetilde{h}_i$ is well-defined. The class $\eta$ is represented by the cocycle which takes value  $m\widetilde{n}_{ij}$ on $U_{ij}$, where $\widetilde{n}_{ij}$ are local sections to $\mathrm{R}^1\pi_*\mathbb{R}$ that represents $\mathcal{N}_i\otimes\mathcal{N}_j^{\otimes-1}$.
\item The class $\delta(\eta)$ is calculated by tracing back the Bockstein sequence. Lifting the local sections $m\widetilde{n}_{ij}$ in $\mathbb{R}^1\pi_*(\mathbb{Z}/m\mathbb{Z})$ to $m\widetilde{n}_{ij}$ in $\mathbb{R}^1\pi_*\mathbb{Z}$, taking boundary then divide by $m$, we get the 2-coboundary that takes value $\widetilde{n}_{ij}-\widetilde{n}_{jk}+\widetilde{n}_{ki}$ on $U_{ijk}$.
\item On the other hand, using the previous uniformization sequence, the class of the torsor $\mathrm{Pic}^{\frac{2g-2}{(m,2g-2)}}_{\mathcal{M}_{g,1}/\mathcal{M}_g}$
 is obtained by the connecting homomorphism $\mathrm{H}^1(\mathcal{M}_g,\mathrm{Pic}^0_{\mathcal{M}_{g,1}/\mathcal{M}_g})\to\mathrm{H}^2(\mathcal{M}_g,\mathrm{R}^1\pi_*\mathbb{Z})$ which acts on the $1$-cocycle $\mathcal{N}_{i}\otimes\mathcal{N}_j^{\otimes-1}$ by lifting the section in $\mathrm{Pic}^0_{\mathcal{M}_{g,1}/\mathcal{M}_g}$ to $\widetilde{n}_{ij}$ in $\mathrm{R}^1\pi_*\mathbb{R}$ then taking coboundary. This coincides with the class of $\delta(\eta)$. This 2-cocycle represents the class of the Picard torsor $\mathrm{Pic}^{\frac{2g-2}{(m,2g-2)}}_{\mathcal{M}_{g,1}/\mathcal{M}_g}$, since $\mathcal{N}_i$ are line bundles of degree $\frac{2g-2}{(m,2g-2)}$.
\end{itemize}
Therefore we have proved that $F^1/F^2$ is generated by $\mathcal{L}^{\otimes\frac{2g-2}{(m,2g-2)}}$, hence $\mathrm{H}^2(\mathcal{M}_g,\mu_m)=\pi^*\mathrm{H}^2(\mathcal{M}_{g},\mu_m)+\delta(\mathrm{Pic}(\mathcal{M}_{g,1}))$ for all $g\geq3$.
\end{proof}

\subsection{Relation with \cite[Theorem 1.2]{zbMATH01991000}}
By Proposition \ref{vanishing-brauer}, we calculate that $\mathrm{Br}'(\mathcal{M}_{3,1})=\mathbb{Z}/2\mathbb{Z}$ over $\mathbb{C}$. By \cite[Proposition 4.1]{zbMATH07456405}, we have $$\mathrm{Br}'(\mathcal{M}_{3,1})=\mathrm{H}_2(\mathcal{M}_{3,1},\mathbb{Z})_{\mathrm{tors}}=\mathrm{H}_2(\Gamma_3^1,\mathbb{Z})_{\textrm{tors}}.$$ In \cite[Theorem 1.2, Proposition 1.4]{zbMATH01991000}, it is proved that $$\mathrm{H}_2(\Gamma_3,\mathbb{Z})=A_1\oplus\mathbb{Z},\ \mathrm{H}_2(\Gamma_{3,1},\mathbb{Z})=A_2\oplus\mathbb{Z},$$ $$\mathrm{H}_2(\Gamma_3^1,\mathbb{Z})=\mathrm{H}_2(\Gamma_{3,1},\mathbb{Z})\oplus\mathbb{Z}=A_2\oplus\mathbb{Z}^{\oplus2},$$ with $0\subseteq A_2\subseteq A_1\subseteq \mathbb{Z}/2\mathbb{Z}$. Therefore $A_1=A_2=\mathbb{Z}/2\mathbb{Z}$ and \begin{enumerate}
\item $\mathrm{H}_2(\Gamma_{3,1},\mathbb{Z})=(\mathbb{Z}/2\mathbb{Z})\oplus \mathbb{Z}$, $\mathrm{H}_2(\Gamma_3,\mathbb{Z})=( \mathbb{Z}/2\mathbb{Z})\oplus \mathbb{Z}$. \item $\mathrm{H}_2(\Gamma_3^1,\mathbb{Z})=(\mathbb{Z}/2\mathbb{Z})\oplus\mathbb{Z}^{\oplus2}.$
\end{enumerate}

\section*{Acknowledgements}
The author thanks Dingxin Zhang for helpful comments on an earlier version of this work. The author is grateful to Professor Richard Hain for sharing with him the results of Morita and a topological version of the Franchetta conjecture in $g\geq3$ \cite[Theorem 24]{zbMATH01590471}, which encouraged the author to think more about the $g=3$ case.

\bibliographystyle{alpha}
\bibliography{references}

@article{zbMATH01991000,
 author = {Korkmaz, Mustafa and Stipsicz, Andr{\'a}s I.},
 title = {The second homology groups of mapping class groups of orientable surfaces},
 fjournal = {Mathematical Proceedings of the Cambridge Philosophical Society},
 journal = {Math. Proc. Camb. Philos. Soc.},
 issn = {0305-0041},
 volume = {134},
 number = {3},
 pages = {479--489},
 year = {2003},
 language = {English},
 doi = {10.1017/S0305004102006461},
 url = {hdl.handle.net/11511/41376},
 zbMATH = {1991000},
 Zbl = {1040.57012}
}

@article{zbMATH01590471,
 author = {Hain, Richard and Reed, David},
 title = {Geometric proofs of some results of {Morita}},
 fjournal = {Journal of Algebraic Geometry},
 journal = {J. Algebr. Geom.},
 issn = {1056-3911},
 volume = {10},
 number = {2},
 pages = {199--217},
 year = {2001},
 language = {English},
 zbMATH = {1590471},
 Zbl = {0986.14017}
}

@misc{arXiv:2509.09661,
 author = {Andrea Di Lorenzo},
 title = {Cohomological invariants of {$\mathcal{M}_{3,n}$} via level structures},
 year = {2025},
 howpublished = {Preprint, {arXiv}:2509.09661 [math.{AG}] (2025)},
 url = {https://arxiv.org/abs/2509.09661},
 arXiv = {arXiv:2509.09661}
}

@book{zbMATH05798333,
 author = {Arbarello, Enrico and Cornalba, Maurizio and Griffiths, Phillip A.},
 title = {Geometry of algebraic curves. {Volume} {II}. {With} a contribution by {Joseph} {Daniel} {Harris}},
 fseries = {Grundlehren der Mathematischen Wissenschaften},
 series = {Grundlehren Math. Wiss.},
 issn = {0072-7830},
 volume = {268},
 isbn = {978-3-540-42688-2; 978-3-540-69392-5},
 year = {2011},
 publisher = {Berlin: Springer},
 language = {English},
 doi = {10.1007/978-3-540-69392-5},
 zbMATH = {5798333},
 Zbl = {1235.14002}
}

@article{zbMATH08105076,
 author = {Di Lorenzo, A. and Pirisi, R.},
 title = {The {Brauer} groups of moduli of genus three curves, abelian threefolds and plane curves},
 fjournal = {Compositio Mathematica},
 journal = {Compos. Math.},
 issn = {0010-437X},
 volume = {161},
 number = {7},
 pages = {1664--1697},
 year = {2025},
 language = {English},
 doi = {10.1112/S0010437X25007481},
 zbMATH = {8105076},
 Zbl = {1578.14023}
}

@article{zbMATH07275229,
 author = {Antieau, Benjamin and Meier, Lennart},
 title = {The {Brauer} group of the moduli stack of elliptic curves},
 fjournal = {Algebra \& Number Theory},
 journal = {Algebra Number Theory},
 issn = {1937-0652},
 volume = {14},
 number = {9},
 pages = {2295--2333},
 year = {2020},
 language = {English},
 doi = {10.2140/ant.2020.14.2295},
 zbMATH = {7275229},
 Zbl = {1459.14007}
}

@incollection{zbMATH06492666,
 author = {Wahl, Nathalie},
 title = {Homological stability for mapping class groups of surfaces},
 booktitle = {Handbook of moduli. Volume III},
 isbn = {978-1-57146-259-6; 978-1-57146-265-7},
 pages = {547--583},
 year = {2013},
 publisher = {Somerville, MA: International Press; Beijing: Higher Education Press},
 language = {English},
 zbMATH = {6492666},
 Zbl = {1322.57016}
}

@book{zbMATH03674235,
 author = {Milne, J. S.},
 title = {{\'E}tale cohomology},
 fseries = {Princeton Mathematical Series},
 series = {Princeton Math. Ser.},
 volume = {33},
 year = {1980},
 publisher = {Princeton University Press, Princeton, NJ},
 language = {English},
 zbMATH = {3674235},
 Zbl = {0433.14012}
}

@article{zbMATH07456405,
 author = {Fringuelli, Roberto and Pirisi, Roberto},
 title = {The {Brauer} group of the universal moduli space of vector bundles over smooth curves},
 fjournal = {IMRN. International Mathematics Research Notices},
 journal = {Int. Math. Res. Not.},
 issn = {1073-7928},
 volume = {2021},
 number = {18},
 pages = {13609--13644},
 year = {2021},
 language = {English},
 doi = {10.1093/imrn/rnz300},
 url = {hdl.handle.net/11573/1572649},
 zbMATH = {7456405},
 Zbl = {1487.14050}
}

@article{zbMATH02078578,
 author = {Schr{\"o}er, Stefan},
 title = {The strong {Franchetta} conjecture in arbitrary characteristics.},
 fjournal = {International Journal of Mathematics},
 journal = {Int. J. Math.},
 issn = {0129-167X},
 volume = {14},
 number = {4},
 pages = {371--396},
 year = {2003},
 language = {English},
 doi = {10.1142/S0129167X03001752},
 zbMATH = {2078578},
 Zbl = {1059.14058}
}

@book{zbMATH07384449,
 author = {Colliot-Th{\'e}l{\`e}ne, Jean-Louis and Skorobogatov, Alexei N.},
 title = {The {Brauer}-{Grothendieck} group},
 fseries = {Ergebnisse der Mathematik und ihrer Grenzgebiete. 3. Folge},
 series = {Ergeb. Math. Grenzgeb., 3. Folge},
 issn = {0071-1136},
 volume = {71},
 isbn = {978-3-030-74247-8; 978-3-030-74248-5},
 year = {2021},
 publisher = {Cham: Springer},
 language = {English},
 doi = {10.1007/978-3-030-74248-5},
 zbMATH = {7384449},
 Zbl = {1490.14001}
}
\end{document}